\documentclass{article}

\usepackage{amsfonts}
\usepackage{enumerate}
\usepackage{amssymb}
\usepackage{amsthm}
\usepackage{amsmath}
\usepackage{comment}
\usepackage{tikz}
\usepackage{url}
\usepackage{hyperref}

\usetikzlibrary{arrows,decorations.pathmorphing,decorations.pathreplacing,positioning,shapes.geometric,shapes.misc,decorations.markings,decorations.fractals,calc,patterns}
\tikzset{>=stealth',
         cvertex/.style={circle,draw=black,inner sep=1pt,outer sep=3pt},
         vertex/.style={circle,fill=black,inner sep=1pt,outer sep=3pt},
         star/.style={circle,fill=yellow,inner sep=0.75pt,outer sep=0.75pt},
         tvertex/.style={inner sep=1pt,font=\scriptsize},
         gap/.style={inner sep=0.5pt,fill=white}}

\newcommand{\ZZ}{\mathbb{Z}}
\newcommand{\NN}{\mathbb{N}}
\newcommand{\CC}{\mathbb{C}}
\newcommand{\QQ}{\mathbb{Q}}

\newcommand{\FF}{\mathbb{F}}
\newcommand{\Es}{E_{p-1}^*}

\newcommand{\FFc}{\bar{\mathbb{F}}}

\newcommand{\Spec}{\mbox{Spec}}

\newcommand{\CL}{\mathcal{L}}
\newcommand{\CM}{\mathcal{M}}
\newcommand{\CO}{\mathcal{O}}
\newcommand{\CN}{\mathcal{N}}

\newcommand{\GZ}{\Gamma_0}
\newcommand{\GI}{\Gamma_1}

\title{Generators and relations of the graded algebra of modular forms}
\author{Nadim Rustom}

\newtheorem{thm}{Theorem}[section]
\newtheorem{rem}[thm]{Remark}
\newtheorem{conj}{Conjecture}
\newtheorem{cor}[thm]{Corollary}
\newtheorem{lem}[thm]{Lemma}

\newtheorem{defn}{Definition}
\newtheorem{algo}{Algorithm}

\begin{document}
\maketitle

\begin{abstract}
We give bounds on the degree of generators for the ideal of relations of the graded algebras of modular forms with coefficients in $\QQ$ over congruence subgroups $\GZ(N)$ for $N$ satisfying some congruence conditions and for $\GI(N)$. We give similar bounds for the graded $\ZZ[\frac{1}{N}]$-algebra of modular forms on $\GI(N)$ with coefficients in $\ZZ[\frac{1}{N}]$. For a prime $p \geq 5$, we give a lower bound on the highest weight appearing in a minimal list of generators for $\GZ(p)$, and we identify a set of generators for the graded algebra $M(\GZ(p),\ZZ)$ of modular forms over $\GZ(p)$ with coefficients in $\ZZ$, showing that, in contrast to the cases studied in \cite{rustom}, this weight is unbounded. We generalize a result of Serre concerning congruences between modular forms over $\GZ(p)$ and $SL_2(\ZZ)$, and use it to identify a set of generators for $M(\GZ(p),\ZZ)$, and we state two conjectures detailing further the structure of this algebra. Finally we provide computations concerning the number of generators and relations for each of these algebras, as well as computational evidence for these conjectures. 
\end{abstract}

\section*{Introduction}
Fix a congruence subgroup $G$ which is either of the form $\GZ(N)$ or $\GI(N)$ for some $N$, and $A$ a subring of $\CC$. A modular form on $G$ has a $q$-expansion:
\[ f(z) = \sum_{i \geq 0} a_i q^i \]
where $q = e^{2i\pi z}$ and $a_i \in \CC$ for all $i\geq 0$. If $A$ is a subring of $\CC$, and if $f$ is a modular form of level $G$ and weight $k$ such that:
\[ f(z) = \sum_{i \geq 0} a_i q^i \in A[[q]]\]
then we say that $f$ has coefficients in $A$. The set of modular forms on $G$ and coefficients in $A$ is an $A$-module, which we denote by $M_k(G,A)$. In the case where $A = \CC$, we write simply $M_k(G) = M_k(G,\CC)$. When $G = SL_2(\ZZ)$ is the full modular group, we say $f$ is of level $1$.\\
The set $\{f \in M_k(G,A): k = 0 \mbox{ or } k \geq 2\}$ generates a graded $A$-module which we denote by $M(G,A)$, and which can be written as the direct sum:
\[M(G,A) = M_0(G,A) \oplus \left( \bigoplus_{k \geq 2} M_k(G,A)\right)\]
where $M_0(G,A) = A$. When $A = \CC$, we write $M(G) = M(G,\CC)$. When there is no ambiguity, we may drop the congruence group and the ring of coefficients from the notation.\\
If $A = \FF_p$, then by $M_k(G,\FF_p)$ we mean the formal mod $p$ reductions of the $q$-expansions of modular forms in $M_k(G,\ZZ)$. Thinking of modular forms as $q$-expansions, we write:
\[M_k(G,\FF_p) = M_k(G,\ZZ)\otimes \FF_p. \]
We then define $M(G,\FF_p)$ similary as:
\[M(G,\FF_p) = M_0(G,\FF_p) \oplus \left( \bigoplus_{k \geq 2} M_k(G,\FF_p)\right).\]
When $G = \GI(N)$ with $N \geq 5$, and $k \geq 2$, we have a more conceptual interpretation of mod $p$ modular forms; see Section 1 for further clarification. The reason we omit modular forms of weight $1$ is to avoid the difficulty arising from the existence of modular forms mod $p$ of weight $1$ which do not lift to characteristic 0, and because in practice it is more difficult to do computations with weight 1 modular forms. \\ 
For a congruence group $G$ and a subring $A$ of $\CC$, the graded $A$-algebra $M(G,A)$ is finitely generated (see \cite{DR}, Th\'eor\`eme 3.4). Thus there exists an integer $n \geq 2$ such that the smallest graded $A$-subalgebra of $M(G,A)$ containing:
\[M_0(G,A) \oplus \left( \bigoplus_{k = 2}^{n} M_k(G,A)\right)\]
is the whole algebra $M(G,A)$. For any such $n$, we say that $M(G,A)$ is generated in weight at most $n$, and the smallest such $n$ is called the generating weight of $M(G,A)$. \\
In \cite{rustom}, we studied the graded algebras of modular forms of various levels and over subrings $A$ of $\CC$. Our main theorem in \cite{rustom} was that for $N \geq 5$, the algebra $M(\GI(N),\ZZ[\frac{1}{N}])$ is generated in weight at most $3$. The key idea was to use a result that first appeared in Mumford's paper \cite{mumford} concerning invertible sheaves on algebraic curves, and apply it to modular curves over finite fields $\FF_p$ for all $p \nmid N$, to conclude the result over $\ZZ[\frac{1}{N}]$. To deal with mod $p$ modular forms, we appealed to the existence of a fine moduli scheme classifying elliptic curves with $\GI(N)$ structure when $N \geq 5$. For levels $\GZ(N)$ (for $N$ satisfying some congruence conditions) and over $\CC$ (equivalently, over $\QQ$), we proved a similar result, where the generating weight is now $6$.\\
This article is divided into three parts. In the first part, we deal with the ideal of relations. Suppose that $\{g_1,\cdots, g_r\}$ is a minimal set of generators of $M(G,A)$. Then one can define a homomorphism of graded $A$-algebras:
\[\Phi : A[x_1,\cdots,x_r] \mapsto M(G,A)\]
given by $\Phi(x_i) = g_i$ for $1 \leq i \leq r$, where $A[x_1,\cdots,x_r]$ is a weighted polynomial ring, each $x_i$ receiving the weight of $g_i$. The ideal $I = \ker \Phi$ is finitely generated, and it is the ideal of relations. We say that the ideal of relations of $M(G,A)$ (always with respect to a minimal set of generators) is generated in degree at most $n$ if one can pick generators $\{r_1,\cdots,r_m\}$ of $I$ each lying in degree at most $m$. Again using the tools in \cite{mumford}, we provide bounds on the degree of generators of the ideal of the relations that exist between elements of a minimal set of generators of the algebra. More precisely, we prove:\\

\textbf{Theorem \ref{relthm}. } \textit{Choosing a minimal set of generators for $M = M(\Gamma,\QQ)$, the ideal of relations is generated:
        \begin{itemize}
            \item in degree at most $6$ when $\Gamma = \Gamma_1(N)$ for $N \geq 5$, or
            \item  in degree at most $10$ when $\Gamma = \Gamma_0(N)$ for $N$ satisfying:
            \[N \equiv 0 \pmod{4} \mbox{ or } N \equiv 0 \pmod{p}, \mbox{ } p \equiv 3 \pmod{4} \]
             and:
            \[N \equiv 0 \pmod{9} \mbox{ or } N \equiv 0 \pmod{p}, \mbox{ } p \equiv 5 \pmod{6}. \]
        \end{itemize}}
The main result of this section concerns the ideal of relations for modular forms of level $\GI(N)$ and with coefficients in $\ZZ[\frac{1}{N}]$, and we prove it using the existence of the fine moduli scheme. Precisely, we prove:\\

\textbf{Theorem \ref{relthmZZ}. } \textit{Let $N \geq 5$. Choosing a minimal set of generators for $M = M(\GI(N),\ZZ[\frac{1}{N}])$, the ideal of relations is generated in degree at most $6$. }\\\\
In the second part, we turn our attention to modular forms of level with coefficients in $\ZZ$. Here we choose to deal with the easiest case, the modular forms of level $\GZ(p)$, although we expect that our method generalizes to other congruence subgroups. First we look at what the generating weight can be, and we prove that, when restricting to coefficients in $\ZZ$, there is no longer a uniform upper bound on the generating weight independently of the level. More explicitly, we prove:\\

\textbf{Theorem \ref{genZ}. }\textit{Let $N \geq 5$ and let $p \geq 5$ be a prime which divides $N$ exactly once. Then any set of generators for $M(\Gamma_0(N),\ZZ)$ contains a form of weight $p-1$. In particular, the generating weight of $M(\Gamma_0(N),\ZZ)$ is at least $p-1$.}\\\\
Next, we proceed to identify a set of generators for $M(\GZ(p), \ZZ)$. This set consists of the $T$-form $T \in M_{p-1}(\GZ(p),\ZZ)$ appearing in \cite{rustom}, given by:

\[ T(z) := \left(\frac{\eta(pz)^p}{\eta(z)}\right)^2, \]

and the subset $S$ of those modular forms with coefficients in $\ZZ$ whose $p$-adic valuation at the cusp $0$ are not too negative. We recall here that the $T$-form played a fundamental role in the algorithm we developed in \cite{rustom} based on the work of Scholl in \cite{scholl}. In particular, for $f = \sum_{n\geq 0} a_n q^n$, we define the $p$-adic valuation of $f$ as: 

\[ v_p(f) = \inf\{v_p(a_n) : n \geq 0\}, \] and we show:\\

\textbf{Theorem \ref{gensZZ}. } \textit{
Let $S = \{f \in M(\GZ(p),\ZZ) : v_p(\tilde{f})\geq 0 \}$. Then $M(\GZ(p), \ZZ)$ is generated by $T$ and the set $S$.} \\\\
In order to prove that these modular forms generate $M(\GZ(p), \ZZ)$, we first prove a generalization of a result of Serre appearing in \cite{serre} concerning congruence relations between modular forms on $\GZ(p)$ and on $SL_2(\ZZ)$. In that paper, Serre proves that for a modular form $f$ of level $p$ and of weight $k$, there is a modular form $g$ of level $1$ and of weight $k' \geq k$ such that $f \equiv g \pmod{p}$. Now for $f \in M_k(\GZ(N),\ZZ)$, and an odd prime $p$ dividing $N$ exactly once, set $\tilde{f} = p^{k/2} f|W_p^{N}$, where $f|W_p^N$ is the image of $f$ under the Atkin-Lehner involution associated to $p$ (see Section \ref{lowerbound}). Our generalization is formulated as follows:\\

\textbf{Theorem \ref{congthm}. }Let $p \geq 5$, $f \in M_k(\GZ(p),\ZZ)$ with $v_p(f) = 0$ and $v_p(\tilde{f}) = k + a$. Then there exists $g \in M_{k-a(p-1)}(SL_2(\ZZ),\ZZ)$ such that $f \equiv g \pmod{p}$. \\\\
Note that whenever $f$ satisfies the conditions of this theorem, we have $k \geq a(p-1)$ by the bounds on $v_p(\tilde{f})$ given in Proposition 3.20 of \cite{DR}. Serre's result covers the case of the theorem where $a \leq 0$, so that the weight we get for the level $1$ modular form $g$ is $k -a(p-1)\geq k$, and he proves this using an elementary trace argument. Our generalization shows that even if $a > 0$, such a congruence holds, and one can actually pick the level 1 form $g$ in weight $k' = k -a(p-1) < k$. Serre's argument does not generalize in an obvious way, and to prove this generalization, we resort to the intersection theory on the Deligne-Mumford stacks classifying elliptic curves with $\GZ(p)$ structure as studied in \cite{DR}. We provide a brief summary of the main notions of this intersection theory. Finally we state two conjectures (Conjectures \ref{G0conj1} and \ref{G0conj2}) concerning the generating weight of the subalgebra generated by $S$.\\
The results of this paper along with those in \cite{rustom} allow us to write down an explicit algorithm that fully determines the structure of algebra $M(G,A)$ with a minimal set of generators and relations (when $A$ is a PID). The last part of this paper comprises the results of computations carried out using the algorithm and contains the structure of this algebra for various $G$ and $A$. It also contains the results of computations that support Conjecture \ref{G0conj2}.\\\\
\textbf{Acknowledgements. } The author wishes to thank Jim Stankewicz for helpful discussions and clarifications he offered regarding the paper of Deligne and Rapoport (\cite{DR}) and their theory of the moduli stacks.

\subsection*{Notation, definitions and conventions}
Here we add futher definitions, notation, and conventions to the ones given above in the introduction. The standard Eisenstein series of level 1 and weight $k$, normalized so that the constant term in the $q$-expansion is $1$, is denoted by $E_k$. The unique cusp form of level 1 and weight $12$ (normalized to have a leading term $q$ in the $q$-expansion) is denoted by $\Delta$. The Dedekind eta function $\eta$ is defined by the product:
\[ \eta(z) = e^{\frac{2i\pi z}{24}} \prod_{n=1}^{\infty}(1-q^n)\]
and is a modular form of weight $\frac{1}{2}$. The modular forms $\eta$ and $\Delta$ are related by:
\[ \eta(z)^{24}= \Delta(z). \]
We recall the defintion of a $T$-form, first defined in \cite{scholl}, which played a key role in our work appearing in \cite{rustom}.
\begin{defn}\label{deftform}For a congruence group $G$ and a ring $A \subset \CC$, a modular form $T \in M_k(G,A)$ is called a $T$-form if it satisfies the following conditions:
\begin{itemize}
\item $T$ only vanishes at the cusp $\infty$, and
\item the $q$-expansion of $T$ lies in $A[[q]]$ and the $q$-expansion of $T^{-1}$ lies in $A((q))$.
\end{itemize}
\end{defn}
If $f \in M_k(G)$, and $\gamma \in GL_2(\QQ)$, where:
\[\gamma = \begin{pmatrix} a & b \\ c & d \end{pmatrix}, \] we define the action of $\gamma$ on $z \in \CC$ by:
\[\gamma \cdot z = \frac{az + b}{cz + d},\] and the operator $-|_k \gamma : f \mapsto f|_k\gamma$ by:
\[(f|_k\gamma)(z) = (\det \gamma)^{k/2} \left( cz + d \right)^{-k} f(\gamma \cdot z). \]
Note that for a fixed level $G$, the collection of operators $\{-|_k\gamma : k \in \NN\}$ defines a graded operator on $M(G)$, which we denote simply by $-|\gamma : f \mapsto f|\gamma$.

\section{Relations}

Let $X$ be a be a smooth, geometrically connected algebraic curve of genus $g$ over a perfect field $k$. Let $\CL$, $\CM$, and $\CN$ be three invertible sheaves on $X$. We have the following exact sequence:

\[0 \rightarrow R(\CL,\CM) \rightarrow H^0(X,\CL)\otimes H^0(X,\CM) \xrightarrow{\mu} H^0(X,\CL\otimes \CM) \rightarrow S(\CL,\CM) \rightarrow 0 \]
where $\mu$ is the natural multplication map: $\sum f_i \otimes g_i \mapsto \sum f_i g_i$, $R(\CL,\CM)$ and $S(\CL,\CM)$ are respectively its kernel and cokernel. In \cite{rustom}, we studied the generators of graded algebras of modular forms, using the first of the following results that first appear in \cite{mumford}.

\begin{lem}\label{mumlema}
Let $\CL$, $\CM$, and $\CN$ be as above. 
\begin{enumerate}
\item If $\deg \CL \geq 2g+1$ and $\deg \CM \geq 2g$, then $\mu$ is surjective. 
\item The natural map:
\[ R(\CL,\CM)\otimes H^0(X,\CN) \rightarrow R(\CL\otimes \CN, \CM) \]
mapping $(\sum f_i\otimes g_i)\otimes h \mapsto \sum (f_i h)\otimes g_i$ is surjective if $\deg \CL \geq 3g+1$, and $\min\{\deg \CM,\deg \CN\} \geq 2g+2$. 
\end{enumerate}
\end{lem}

We recall the following fact proven in \cite{rustom}.
\begin{cor} \label{gencorQ} $M(\Gamma,\QQ)$ is generated:
\begin{enumerate}
\item in weight at most $3$, when $\Gamma = \Gamma_1(N)$ for $N \geq 5$, or
\item in weight at most $6$, when $\Gamma = \Gamma_0(N)$ for $N$ satisfying the following congruence conditions:
            \[N \equiv 0 \pmod{4} \mbox{ or } N \equiv 0 \pmod{p}, \mbox{ } p \equiv 3 \pmod{4} \]
             and:
            \[N \equiv 0 \pmod{9} \mbox{ or } N \equiv 0 \pmod{p}, \mbox{ } p \equiv 5 \pmod{6} \]
\end{enumerate}
\end{cor}

We make the following additional observations.

\begin{rem}\label{remweights}\indent 
\begin{enumerate}
\item In \cite{rustom}, we used a weaker version of Lemma \ref{mumlema}, resulting in the upper bound $6$ on the weight in part (2) of Corollary \ref{gencorQ}. In fact, Lemma \ref{mumlema} allows us to show, using the same proof as in \cite{rustom}, that the bound in part(2) of Corollary \ref{gencorQ} can actually be taken to be $4$. 
\item It follows from the first part of Corollary \ref{gencorQ} that a minimal set of generators of $M(\GI(N),\QQ)$ must be the union of bases of the spaces $M_2(\GI(N),\QQ)$ and $M_3(\GI(N),\QQ)$, since no weight $2$ forms can give rise to forms in weight $3$. 
\end{enumerate}
\end{rem}

Consider the graded $\QQ$-algebra $M = M(\Gamma, \QQ)$ and pick a minimal set of generators $\{g_1,\cdots,g_n\}$ for it. This provides us with a map of graded algebras:
\[ \Phi: A \rightarrow M \]
\[ x_i \mapsto g_i \]
where $A = \QQ[x_1,\cdots,x_n]$ is the polynomial algebra where each $x_i$ is given the weight of $g_i$. We denote by $A_k$ the $\QQ$-vector space spanned by degree $k$ polynomials. We wish to examine generators of the homogeneous ideal $\ker \Phi$, which is the ideal of relations. In \cite{wagreich1} and \cite{wagreich2}, Wagreich describes the generators of the ideal of relations of $M$ when $M$ is generated by at most 4 generators. Our result is valid regardless of the number of generators involved.

\begin{thm} \label{relthm}
Choosing a minimal set of generators for $M = M(\Gamma,\QQ)$, the ideal of relations is generated:
        \begin{enumerate}
            \item in degree at most $6$ when $\Gamma = \Gamma_1(N)$ for $N \geq 5$, 
            \item  in degree at most $10$ when $\Gamma = \Gamma_0(N)$ for $N$ satisfying:
            \[N \equiv 0 \pmod{4} \mbox{ or } N \equiv 0 \pmod{p}, \mbox{ } p \equiv 3 \pmod{4} \]
             and:
            \[N \equiv 0 \pmod{9} \mbox{ or } N \equiv 0 \pmod{p}, \mbox{ } p \equiv 5 \pmod{6} \]
        \end{enumerate}
\end{thm}
\begin{proof}
For (1), as per Remark \ref{remweights}, if we choose a minimal set of generators, then there are no relations in degrees $2$ or $3$. Let $P \in A$ be a homogeneous polynomial of degree $k \geq 7$, representing a relation in $M$ in weight $k \geq 7$. As explained in \cite{rustom}, we have a line bundle $\CL$ on the modular curve $X = X_1(N)$ of genus $g$ such that:
\[ M_k = H^0(X, \CL^{\otimes k}) \] 
and $\deg \CL \geq g+1$ (when $N \geq 5$). \\
Let $a = 5$ if $k = 7$, and $a = 6$ otherwise. The polynomial $P$ is the sum of homogeneous monomials, each of degree $k$. Since $k \geq 7$, each of these monomials is divisible by a monomial of degree $a$. Let $\{v_1,\cdots,v_n\}$ be a basis of $A_{k-a}$. Thus we can write:
\[ P = \sum_{i = 1}^n Q_i v_i \] where for all $i$, $Q_i$ is a homogeneous polynomial of degree $a$. \\
For some $m$, and possibly after a reordering of the $v_i$'s, the set $\{\Phi(v_1),\cdots,\Phi(v_m)\}$ is a basis of $M_{k-a}$. This means that for every $j$ such that $m+1 \leq j \leq n$ there are constants $\alpha_{1,j},\cdots,\alpha_{m,j}$ such that:
\[ \Phi(v_j) = \sum_{i=1}^{m} \alpha_{i,j} \Phi(v_i).\]
Hence letting 
\[G_j := v_j - \sum_{i=1}^{m} \alpha_{i,j} v_i\]
for $m+1 \leq j \leq n$, we have $\Phi(G_j) = 0$, i.e., the $G_j$'s are relations in weight $k-a$. Now if for $1 \leq i \leq m$ we set:
\[Q_i' := Q_i + \sum_{j=m+1}^n \alpha_{i,j} Q_j, \] 
we can rewrite $P$ as:
\[ P = \sum_{i=1}^m Q_i' v_i + \sum_{j=m+1}^{n} G_j Q_j, \]
where, since $\Phi(P) = 0$ and $\Phi(\sum_{j=m+1}^{n} G_j Q_j) = 0$, we have $\Phi(\sum_{i=1}^m Q_i' v_i) = 0$.
We see then that $\sum_{i=1}^m Q_i' v_i$ must be represented in $R(\CL^{\otimes(k-a)},\CL^{\otimes a})$, that is:
\[\sum_{i=1}^m \Phi(Q_i')\otimes \Phi(v_i) \in R(\CL^{\otimes(k-a)},\CL^{\otimes a}).\] We have then the following diagram:

\[\begin{tikzpicture}[xscale=3,yscale=-1.5]
 \node (Z) at (0,0) {$0$};
 \node (R) at (0.65,0) {$R(\CL^{\otimes a},\CL^{\otimes(k-a)})$};
 \node (M) at (2.1,0) {$H^0(X,\CL^{\otimes a})\otimes H^0(X,\CL^{\otimes(k-a)})$};
 \node (N) at (3.4,0) {$H^0(X,\CL^{\otimes k}) $};
 \node (R2) at (0.65,1) {$R(\CL^{\otimes 3},\CL^{\otimes (k-a) })\otimes H^0(X,\CL^{\otimes (a-3)})$};

 \draw [->] (Z) --  (R);
 \draw [->] (R) -- (M);
 \draw [->] (M) -- (N);
 \draw [->] (R2) -- node [left]{$\epsilon$} (R);
\end{tikzpicture}\]
 By Lemma \ref{mumlema}, the map $\epsilon$ is surjective. Thus for each $i$, we can find polynomials $H_{s}$ in degree $3-a+k \leq k-2$, and polynomials $F_{i,s}$ in degree $a-3$, satisfying:
 
 \[\Phi(\sum_{i=1}^m F_{i,s} v_i) = 0\]
 and such that:
\[ \sum_{i=1}^m \Phi(\sum_s H_{s} F_{i,s}) \otimes \Phi(v_i) = \sum_{i=1}^m \Phi(Q_i')\otimes \Phi(v_i), \] so that:
\[ \Phi(\sum_s H_{s} F_{i,s}) = \Phi(Q_i'), \] hence
\[ Q_i' = \sum_s H_{s} F_{i,s} + W_i \]
where for all $i$, $\Phi(W_i) = 0$, i.e. $W_i$ is a relation in weight $a$. Putting all the above together, we get:
\[ P = \sum_{i=1}^m (\sum_s H_{s} F_{i,s})v_i + \sum_{i=1}^{m} W_i v_i + \sum_{j=m+1}^{n} G_j Q_j \]
so $P$ can be written in terms of relations of degrees $k-(a-3)$, $a$, and $k-a$, which are all $\leq k-2$. \\

For (2), suppose that $N$ satisfies the given congruence conditions. When the genus $g$ of the modular curve $X = X_0(N)$ is 0, the statement follows from the results of \cite{TS11}. So suppose that the modular curve $X = X_0(N)$ is of genus $g \geq 1$. We have a line bundle $\CL$ on $X$ such that for even $k$:
\[ M_k = H^0(X,\CL^{\otimes k/2}) \]
(since the conditions on $N$ remove any elliptic points and irregular cusps), and $\deg \CL \geq 2g$. We can then repeat the same argument as above since for even $k \geq 10$, we have the following diagram: 

\[\begin{tikzpicture}[xscale=3,yscale=-1.5]
 \node (Z) at (-0.35,0) {$0$};
 \node (R) at (0.40,0) {$R(\CL^{\otimes 4},\CL^{\otimes(k/2-4)})$};
 \node (M) at (2,0) {$H^0(X,\CL^{\otimes 4})\otimes H^0(X,\CL^{\otimes(k/2-4)})$};
 \node (N) at (3.4,0) {$H^0(X,\CL^{\otimes k/2}) $};
 \node (R2) at (0.40,1) {$R(\CL^{\otimes 2},\CL^{\otimes (k/2-4) })\otimes H^0(X,\CL^{\otimes 2})$};

 \draw [->] (Z) --  (R);
 \draw [->] (R) -- (M);
 \draw [->] (M) -- (N);
 \draw [->] (R2) -- node [left]{$\epsilon$} (R);
\end{tikzpicture}\]

where the map $\epsilon$ is surjective by Lemma \ref{mumlema}. Thus any relation of degree $\geq 12$ can be written as a combination of relations of degrees $k-4$, $8$, and $k-8$. 
\end{proof}
We now turn to the case of modular forms of level $\GI(N)$ over $\ZZ[\frac{1}{N}]$. First we consider the situation in positive characteristic. Note that while in characteristic 0 any relation must be a homogeneous polynomial in the generators (see \cite{miyake}, Lemma 2.1.1), in positive characteristic one might have non-homogeneous relations (for example, the famous $E_{p-1} \equiv 1 \pmod{p}$). Here we restrict our attention to homogeneous relations. Recall that (cf. \cite{G90}) when $N \geq 5$, the functor representing generalized elliptic curves over $\ZZ[\frac{1}{N}]$ with a choice of a point of exact order $N$ is representable by a fine moduli $\ZZ[\frac{1}{N}]$-scheme $X_1(N)$. On this scheme there is an invertible sheaf, which we denote by $\omega$. The modular forms of level $\GI(N)$ over $\ZZ[\frac{1}{N}]$ are just global sections of tensor powers of this invertible sheaf:
\[ M_k(\GI(N), \ZZ[\frac{1}{N}]) = H^0 ( X_1(N), \omega^{\otimes k}). \]
When $p \nmid N$, the scheme $X_1(N)$ admits a good reduction 
\[X_1(N)_{\FF_P} := X_1(N) \otimes \FF_p,\]
and the invertible sheaf $\omega$ also pulls back to an invertible sheaf $\omega_{\FF_p}$, and by base change theorems, we have the following identification (for $k \geq 2)$:
\[ M_k(\GI(N),\FF_p) = H^0(X_1(N)_{\FF_p}, \omega_{\FF_p}^{\otimes k}). \]

\begin{lem}\label{modprels}
Let $N \geq 5$, and $p \nmid N$ be a prime. Choosing a minimal set of generators of $M(\GI(N),\FF_p)$, the homogeneous relations are generated in degree at most $6$.
\end{lem}
\begin{proof}
As $\FF_p$ is perfect, and $\deg \omega_{\FF_p} = \deg \omega$ (this is shown in the proof of Proposition 1 in \cite{rustom}), the argument in Theorem \ref{relthm} goes through unchanged. 
\end{proof}
As above, suppose we have a minimal set of generators $\{g_1,\cdots,g_n\}$ for $M = M(\GI(N),\ZZ[\frac{1}{N}])$, and let $A = \ZZ[\frac{1}{N}][x_1,\cdots,x_n]$ be the weighted polynomial algebra where each $x_i$ is given the weight of $g_i$. Again, the homogeneous ideal $\ker \Phi$ is finitely generated. This provides us with a map of graded algebras:
\[ \Phi: A \rightarrow M \]
\[ x_i \mapsto g_i. \]
So for every weight $k$, there is a short exact sequence:
\[\begin{tikzpicture}[xscale=3,yscale=-1.5]
 \node (Z1) at (0,0) {$0$};
 \node (R) at (0.65,0) {$\ker(\Phi)_k$};
 \node (A) at (1.3,0) {$A_k$};
 \node (M) at (2,0) {$M_k$};
 \node (Z2) at (2.65,0) {$0$};

 \draw [->] (Z1) --  (R);
 \draw [->] (R) -- (A);
 \draw [->] (A) -- node[above]{$\Phi$} (M);
 \draw [->] (M) -- (Z2);
\end{tikzpicture}\] 
where the subscript $k$ indicates weight $k$ homogeneous submodule. For $p\nmid N$, the map $\Phi$ induces a morphism $\bar{\Phi}: A_k \otimes \FF_p \rightarrow M_k \otimes \FF_p$. Since $\FF_p$ is flat over $\ZZ[\frac{1}{N}]$, so we have a commutative diagram with exact rows:
\[\begin{tikzpicture}[xscale=3,yscale=-1.5]
 \node (Z1) at (0,0) {$0$};
 \node (R) at (0.8,0) {$\ker(\Phi)_k $};
 \node (A) at (1.6,0) {$A_k$};
 \node (M) at (2.4,0) {$M_k$};
 \node (Z2) at (3.2,0) {$0$};
 \node (Z1p) at (0,1) {$0$};
 \node (Rp) at (0.8,1) {$\ker(\Phi)_k \otimes \FF_p$};
 \node (Ap) at (1.6,1) {$A_k \otimes \FF_p$};
 \node (Mp) at (2.4,1) {$M_k \otimes \FF_p$};
 \node (Z2p) at (3.2,1) {$0$};
 
 \draw [->] (Z1) --  (R);
 \draw [->] (R) -- (A);
 \draw [->>] (A) -- node[above]{$\Phi$}(M);
 \draw [->] (M) -- (Z2);
  \draw [->] (Z1p) --  (Rp);
 \draw [->] (Rp) -- (Ap);
 \draw [->>] (Ap) -- node[above]{$\bar{\Phi}$}(Mp);
 \draw [->] (Mp) -- (Z2p);
 
 \draw [->] (R) -- node[right]{$\alpha$}(Rp);
 \draw [->>] (A) -- node[right]{$\beta$}(Ap);
 \draw [->>] (M) -- node[right]{$\gamma$}(Mp);
\end{tikzpicture}\]
As the map $\alpha$ in the above diagram is surjective, this shows that any mod $p$ homogeneous relation of degree $k$ between the generators must be the reduction mod $p$ of a relation in characteristic $0$. That is, we have shown the following lemma:
\begin{lem}\label{relred} The mod $p$ relations of degree $k$ between the generators are precisely the elements of $\ker(\Phi)_k \otimes \FF_p$, the reductions mod $p$ of relations of degree $k$ in characteristic 0.
\end{lem}
We now proceed to prove:
\begin{thm}\label{relthmZZ}
Let $N \geq 5$. Choosing a minimal set of generators for $M = M(\GI(N),\ZZ[\frac{1}{N}])$, the ideal of relations is generated in degree at most $6$.
\end{thm}
\begin{proof}
We argue as in the proof of Theorem 1 in \cite{rustom}. Let $R_1,\cdots,R_m$ be the relations that generate $\ker\Phi$ up to degree $6$. By Lemma \ref{relred} and Lemma \ref{modprels}, the reductions $\bar{R}_1,\cdots,\bar{R}_m$ generate the homogeneous relations mod $p$. Let $B_0$ be a relation in degree $ k \geq 7$. Suppose for the sake of contradiction that $B_0$ is not in the submodule of $\ker\Phi$ generated by $R_1,\cdots,R_m$. As in the proof of Theorem \ref{relthm}, we can assume $R_1,\cdots,R_m$ have weights in $\{k-(a-3), a, k-a\}$, where $a = 5$ if $k = 7$, and $a = 6$ otherwise. The reduction mod $p$ of $B_0$ can be written as:
\[ \bar{B}_0 = \sum_{i=1}^m \bar{F}^{(0)}_i \bar{R_i} \]
where $\bar{F}^{(0)}_i$ have weights $\geq 2$, so that the mod $p$ modular forms represented by these polynomials have lifts to characteristic $0$. Let $F^{(0)}_i \in \ZZ[\frac{1}{N}][x_1,\cdots,x_n]$ denote a lift of $\bar{F}^{(0)}_i$. By our assumption, 
\[ B_0 - \sum_{i=1}^m F^{(0)}_i R_i \not = 0\]
Then there must be some nonzero polynomial $B_1$ such that:
\[ B_0 - \sum_{i=1}^m F^{(0)}_i R_i = pB_1.\]
Since $B_0$ and $R_i$ are relations of degree $k$, it follows that $B_1$ is itself a relation of degree $k$, and we can repeat the process with $B_1$:
\[B_1 - \sum_{i=1}^m F^{(1)}_i R_i = pB_2\] for some $B_2, F^{(1)}_i \in \ZZ[\frac{1}{N}][x_1,\cdots,x_n]$. Iterating this, we have:
\[ B_0 = \sum_{i=1}^m F_i R_i \]
where $F_i \in \ZZ_{p}[x_1,\cdots,x_n]$. This holds for each $p \nmid N$, so by Lemma 4 of \cite{rustom}, we conclude that for each $i$, $F_i \in \ZZ[\frac{1}{N}][x_1,\cdots, x_n]$, whence a contradiction.
\end{proof}

\section{Modular forms with coefficients in $\ZZ$}
\subsection{The lower bound}\label{lowerbound}
A classical known fact is that $M(SL_2(\ZZ),\CC) = \CC[E_4,E_6]$, and one can easily show the stronger statement that $M(SL_2(\ZZ),\ZZ[\frac{1}{6}]) = \ZZ[\frac{1}{6}][E_4, E_6]$, see Theorem 6.3 in \cite{kilford}. If one wishes to restrict consideration to modular forms with coefficients in $\ZZ$, then we find that $M(SL_2(\ZZ),\ZZ)$ is generated by $E_4, E_6$ and $\Delta$, and in fact (see \cite{delignetate}):
\[M(SL_2(\ZZ),\ZZ) \cong \ZZ[E_4,E_6, \Delta]/(E_4^3 - E_6^2 - 1728\Delta). \]
We see that the generating weight increases from $6$ to $12$ when we pass to the ring $\ZZ$. The purpose of this section is to show that the increase in the generating weight when passing to $\ZZ$ is a general phenomenon.\\
Consider a level $N \geq 1$ and an odd prime $p$ dividing $N$ exactly once. We can then define the Atkin-Lehner involution:
\[ W^N_p = \begin{pmatrix} p & a \\ N & bp  \end{pmatrix} \]
where $a$ and $b$ are any integers such that $\det W_p^N = p^2 b - Na = p$.
First we need the following lemma, due to Kilbourn (see \cite{kilbourn}). This lemma generalizes a result obtained in prime level in \cite{DR}.
\begin{lem}\label{killema}
Let $N \geq 1$ and let $p$ be an odd prime dividing $N$ exactly. Then for all $k \leq p-3$, :
\[ |v_p(f|W^N_p) - v_p(f)| \leq k/2. \]
\end{lem}
For convenience, we will make the following definition.
\begin{defn}\label{definv}
Let $N$ and $p$ be as in Lemma \ref{killema}, $k \geq 0$ and $f \in M_k(\GZ(N),\QQ)$. Then we define the following operator:
\[ \tilde{f} := \tilde{\omega}(f) := p^{k/2} f|W^N_p. \]
\end{defn}
Then we have a corollary of Lemma \ref{killema}:
\begin{cor}\label{kilcor}
Let $N$ and $p$ be as in Lemma \ref{killema}, $0 \leq k \leq p-3$, and $f \in M_k(\GZ(N),\QQ)$. Then:
\[ v_p(\tilde{f}) \geq v_p(f). \]
In particular, if $v_p(f) = 0$, then $v_p(\tilde{f}) \geq 0$. 
\end{cor}
We will prove the following:
\begin{thm}\label{genZ} Let $N \geq 5$ and let $p \geq 5$ be a prime which divides $N$ exactly once. Then any set of generators for $M(\Gamma_0(N),\ZZ)$ contains a form of weight $p-1$. In particular, the generating weight of $M(\Gamma_0(N),\ZZ)$ is at least $p-1$. 
\end{thm}

\begin{proof}
The idea of the proof is to produce a modular form in weight $p-1$ that cannot be written as a polynomial with $\ZZ$ coefficients in modular forms with $\ZZ$ coefficients in weights $< p-1$. Recall (see \cite{rustom}) the $T$-form $T$, given by:
\[ T(z) := \left(\frac{\eta(pz)^p}{\eta(z)}\right)^2 \in M_{p-1}(\GZ(p)) \subset M_{p-1}(\GZ(N))\]
We recall also that both $T$ and $T^{-1}$ have $q$-expansion coefficients in $\ZZ$. It is also obvious that $v_p(T) = 0$. The truth of Theorem \ref{genZ} then clearly follows from the following lemma:
\begin{lem}\label{vptlem}
The form $T$ is not a polynomial with $\ZZ$ coefficients in modular forms with $\ZZ$ coefficients in weights $<p-1$.
\end{lem}

\begin{proof}
We will prove the lemma by showing that $T$ violates the inequality in Lemma \ref{killema}, and that every modular form which is a polynomial in forms of weight $< p -1$ must satisfy the inequality.\\
We define the following matrices:
\[ M_p = \begin{pmatrix} 1 & a \\ \frac{N}{p} & pb  \end{pmatrix}, \]
\[ M_p' = \begin{pmatrix} p & a \\ \frac{N}{p} & b  \end{pmatrix} \]
\[ \gamma = \begin{pmatrix} p & 0 \\ 0 & 1  \end{pmatrix}. \]
We note that $W^N_p = M_p \gamma$, and that $\gamma W^N_p = p M_p'$, so $\gamma W^N_p$ acts on modular forms via the $-|_k$ operator in the same way as $M_p'$ does. We with to compute $v_p(\tilde{T})$ where $\tilde{T} = p^{\frac{p-1}{2}}T|W_p^N$. We have:
\[\tilde{T}(z) = p^{p-1} (Nz + pb)^{-(p-1)} \left(\frac {\eta(M_p' \cdot z)^p}{\eta(W^N_p \cdot z)}\right)^2. \]
By applying the appropriate transformation formulae for the $\eta$ function (see for instance \cite{koehler}), we have:

\[\eta(M_p' \cdot z)^2 = p \nu_\eta(M_p')^2 (Nz + pb) \eta(z)^2, \]

\[\eta(W^N_p \cdot z)^2 = \nu_\eta (M_p)^2 (Nz + pb) \eta(p\cdot z)^2, \]

where $\nu_\eta(-)$ is the eta multiplier. After writing out the multipliers explicitly, we find that:

\[ \tilde{T}(z) = \epsilon p^{-1} \left(\frac{\eta(z)^p}{\eta(pz)}\right)^2 \]
where:
\[ \epsilon = \begin{cases} e^{\left(\frac{2Np - 6p - 2N/p + 6}{24}\right)} & \mbox{ if } N/p \in 2\ZZ \\
e^{\left(\frac{2Np - 6N + 4N/p}{24}\right)} & \mbox{ otherwise }\end{cases}, \]
and $e(z) = e^{2i\pi z}$. It is is show to show that $\epsilon = \pm 1$, and hence that $v_p(\tilde{T}) = -1$. All we need to show is that 
\[ Np - 3p - N/p + 3 \equiv 0 \pmod{6} \mbox{  if } \frac{N}{p} \equiv 0 \pmod{2} \]
and that 
\[ Np - 3N + 2N/p \equiv 0 \pmod{6}  \mbox{  if } \frac{N}{p} \equiv 1 \pmod{2}. \]
Indeed, we have $p \equiv 1 \pmod{2}$ and $p \equiv \pm 1 \pmod{3}$, so if $\frac{N}{p} \equiv 0 \pmod{2}$, then:
\[\begin{cases} Np - 3p - N/p + 3 \equiv 0 \pmod{2} & \mbox{, and} \\
               Np - 3p - N/p + 3 \equiv 0 \pmod{3}.
\end{cases}\] Similarly, if $\frac{N}{p} \equiv 1 \pmod{2}$, then:
\[\begin{cases} Np - 3N + 2N/p \equiv 0  \pmod{2} & \mbox{, and} \\
               Np - 3N + 2N/p  \equiv 0 \pmod{3}.
\end{cases}\]
To finish the proof, note that the operator $\tilde{\omega}$ of Definition \ref{definv} defines an operator on the graded algebra of modular forms:
\[ \widetilde{fg} = \tilde{f} \cdot \tilde{g}, \]
\[ \widetilde{f+g} = \tilde{f}+ \tilde{g}. \]
Suppose now that 
\[T = \sum c_{i_1,\cdots,i_n} g_1^{i_1} \cdots g_n^{i_n}\] where $c_i \in \ZZ$ and $g_i \in M(\GZ(N),\ZZ)$ are modular forms in weights $\leq p-3$. Then: 
\[\tilde{T} = \sum c_{i_1,\cdots,i_n} (\tilde{g_1})^{i_1} \cdots (\tilde{g_n})^{i_n},\]
which would force $v_p(\tilde{T}) \geq 0$ by Corollary \ref{kilcor}, but that contradicts the above computation of $v_p(\tilde{T})$. 
\end{proof}
We have now established Lemma \ref{vptlem}, and Theorem \ref{genZ} follows immediately.
\end{proof}
In \cite{rustom}, we raised the question of the lowest weight in which one could find a $T$-form (see Definition \ref{deftform}) for $\GZ(p)$, and we proved that the lowest weight is either $p-1$ or $\frac{p-1}{2}$. As a corollary to Lemma \ref{vptlem}, we have the following:
\begin{cor}
Let $p \geq 5$. The lowest weight in which one can find a $T$-form for $\GZ(p)$ is $p-1$.
\end{cor}
\begin{proof}
Let $T$ be the $T$-form in weight $p-1$ defined above. Suppose there exists a $T$-form $T'$ of lower weight. By the defining properties of $T$-forms, it follows that $T'$ divides $T$ in the algebra of modular forms, and that $\frac{T}{T'} = T''$ is a $T$-form in weight $< p-1$. Then $T = T' T''$, but this contradicts Lemma \ref{vptlem}. 
\end{proof}

\subsection{Intersection theory on $\CM_{\GZ(p)}$}\label{intheory}
For the convenience of the reader, we summarize here the main definitions regarding the intersection theory on the stacks $\CM_{\GZ(p)}$ studied in \cite{DR}. For our purpose, we only need the theory developed in \cite{DR}. However, the appendix by Brian Conrad in \cite{BDP} contains a more explicit and more general treatement of the intersection theory on such stacks.\\
The stack $\CM_{\GZ(p)}$ is the moduli stack classifying elliptic curves (over $\ZZ$) with a choice of a subgroup of order $p$. These stacks are Deligne-Mumford (DM); the main property of DM stacks that we need is that they admit finite \'etale covers by schemes. Thus one can define sheaves on them as sheaves on the \'etale site, in particular, these stacks are locally ringed, and one can use the intersection theory of schemes to define intersection concepts on the stacks. For the definition and basic properties of these stacks, see \cite{DM}.\\
The stack $\CM_{\GZ(p)}$ is not representable, i.e.\ it is not a scheme, since every pair $(E,C)$ consisting of an elliptic curve $E$ and a subgroup $C$ of order $p$ admits at least one non-trivial automorphism (the involution corresponding to $-1 \in \GZ(p)$). It is regular (\cite{DR},Th\'eor\`eme V.1.16), of dimension 2 (of relative dimension $1$) over $\Spec(\ZZ)$.\\
Let $\CM$ be such a stack, and let $\CL$ be an invertible sheaf on $\CM$. As in \cite{DR}, VI.4.3, the degree of $\CL$ is defined as follows. Suppose that $\CL$ has a rational section $f$. Pick a geometric fiber (for example, say it is $\CM\otimes k$ where $k$ is an algebraically closed field), and at each closed geometric point $x$ of this fiber, define:
\[ \deg_x (f) = \begin{cases} \dim_k \widetilde{O_x}/(f) & \mbox{if } f \mbox{ is regular at } x \\ - \dim_k \widetilde{O_x}/(f^{-1}) & \mbox{ otherwise}\end{cases}\]
where $\widetilde{O_x}$ is the henselian local ring of the fiber at $x$. Then the degree of $\CL$ is defined by:
\[\deg \CL = \sum_x \frac{deg_x(f)}{|Aut(x)|} \]
where $Aut(x)$ is the automorphism group of the elliptic curve represented by the point $x$. This degree is independent of the choice of the fiber. 
\\
A reduced irreducible closed substack of codimension 1 is Cartier (Lemma B.2.2.8 in \cite{BDP}). A Cartier divisor is effective if the ideal sheaf associated to the corresponding closed substack is invertible. If $D$ is an effective Cartier divisor, there is an invertible sheaf $\CO(D)$ associated to it that has a canonical regular global section $s_D$. By regularity of the stack, the henselian local ring at every codimension 1 point is a DVR, thus to every effective Cartier divisor one can associate an effective Weil divisor (i.e. a finite formal integral combination of closed reduced irreducible substacks of co-dimension 1, where all the coefficients are non-negative). For an invertible sheaf $\CL$ on $\CM$, and a global section $s$ of $\CL$ that is non-zero on every connected component of $\CM$, we can associate a Weil divisor $div(f)$ such that there is an isomorphism of sheaves $\CO(div(f)) \cong \CL$. \\
Thus we can identify the concepts of an effective Cartier divisor and an effective Weil divisor. Given an invertible sheaf $\CL$ with a global section $s$ which is non-zero on every connected component, its divisor $D = div(f)$ can be written as the sum of horizontal and vertical divisors. If $N$ is an irreducible component of a geometric fiber, seen as a vertical divisor (by giving it the reduced structure), then one can define the intersection number:
\[ (D,N) := \deg \CO(D)|_N. \]
The degree of $\CL$ can equally be defined as the intersection number of the divisor of a global section with a geometric fibral divisor.\\
The stack classifiying generalized elliptic curves (without a choice of a structure) is denoted $\CM_1$. It is shown in \cite{DR} that the reduction mod $p$, $\CM_{\GZ(p)} \otimes \FF_p$ consists of two copies of $\CM_1 \otimes \FF_p$ glued together at the supersingular points. \\
Similar to the case of the moduli schemes, on each of the stacks $\CM_{\GZ(p)}$ and $\CM_1$, there is an invertible sheaf $\omega$ such that the modular forms (of levels $p$ and $1$ respectively) of weight $k$ can be seen as global sections of $\omega^{\otimes k}$. 

\subsection{Congruences of level $1$ and level $p$ modular forms}
For this section, fix a prime $p \geq 5$. We look at congruence relations between modular forms on $SL_2(\ZZ)$ and on $\GZ(p)$, that is, congruences between their formal $q$-expansions at infinity. In \cite{serre}, Serre proves that every modular form in $M(\GZ(p),\ZZ)$ is $p$-adically of level $1$, that is, if $f \in M_k(\GZ(p),\ZZ)$ for some $k$, then for every integer $i > 0$, there exists an integer $k_i$ and a modular form $f_i \in M_{k_i}(SL_2(\ZZ),\ZZ)$ such that $f \equiv f_i \pmod{p^i}$. In particular one has the following result. Let $v = v_p(\tilde{f})$, and let:
\[ 
    \Es = E_{p-1} - \tilde{E}_{p-1}.
\]
The form $\Es$ has the following properties: $\Es \equiv 1 \pmod{p}$ and $v_p(\widetilde{\Es})=p$. If $v \leq k$, this implies that $tr(f(\Es)^{k-v}) \in M_{kp - v(p-1)}(SL_2(\ZZ))$ is $p$-integral and is congruent to $f$ modulo $p$. Here, $tr$ is the trace operator sending modular forms of level $p$ to modular forms of level $1$; for the definition and properties, see \cite{serre}.\\
When $v > k$, if the above congruence still holds, then we expect to see $f$ mod $p$ in weight $kv - v(p-1) < k$. Since this weight is less than $k$, Serre's trace argument apparently no longer applies. The aim of this section is to show that a similar congruence relation still holds even when the ``expected weight" for $f$ is less than $k$. That is, we have:
\begin{thm}\label{congthm}
Let $p \geq 5$, $f \in M_k(\GZ(p),\ZZ)$ with $v_p(f) = 0$ and $v_p(\tilde{f}) = k + a$. Then there exists $g \in M_{k-a(p-1)}(SL_2(\ZZ),\ZZ)$ such that $f \equiv g \pmod{p}$. 
\end{thm}
\begin{proof}
The case where $a\leq 0$ is covered by Serre's argument in \cite{serre}. We deal here with the case where $a > 0$.The proof relies on Deligne and Rapoport's study of the  stack $\CM_{\GZ(p)}$ in \cite{DR}. The intersection theory on such stacks is summarized in Section \ref{intheory}.\\
The modular form $f$ can be seen as a global section $f \in H^0 (\CM_{\GZ(p)}, \omega^{\otimes k})$. Let $N_1$ and $N_2$ be the two irreducible components of $\CM_{\GZ(p)}\otimes \FF_p$ containing respectively the (reductions of the) cusps $\infty$ and $0$. Then $f$ does not vanish at the generic point of $N_1$, and it vanishes to order $a$ at the generic point of $N_2$. Thus the divisor of $f$ can be written as:
\[ div (f) = D + aN_2\]
where, without loss of generality after multiplying $f$ by a constant of $p$-adic valuation $0$, we can assume\footnote{While $f$ might have some poles along certain vertical (i.e.\ fibral) divisors, it cannot have a pole along a horizontal divisor. This is because any horizontal divisor would meet the generic fiber, and so if $f$ has a pole along a horizontal divisor, then $f$ would have a pole when considered as a modular form over $\CC$, which contradicts the holomorphy of $f$ as a complex function of a complex variable. The vertical components of the divisor of poles correspond to the primes appearing in the denominators of the $q$-expansion of $f$. As these denominators are bounded, one can find a constant $K \in \ZZ$, $K \not \equiv 0 \pmod{p}$ such that $Kf$ has no primes in the denominators of its $q$-expansion except possibly $p$. Multipying by such a constant obviously preserves the $p$-adic valuation of $f$ at both cusps}that $D$ is an effective horizontal Cartier divisor on $\CM_{\GZ(p)}$. As in \cite{DR}, VII.3.19, we can calculate the intersection number $(D,N_1)$ as follows: on $N_1 \cong \CM_1 \otimes \FF_p$, the degree of $\omega$ is $\frac{1}{24}$ (\cite{DR}, VI.4.4.1), hence:
\[ (div(f), N_1) = \frac{k}{24}. \]
The components $N_1$ and $N_2$ intersect transversally at the supersingular points. It follows from \cite{DR} (Th\'eor\`eme V.1.16 and Th\'eor\`eme VI.4.9.1) that:
\[ (N_2, N_1) = \frac{p-1}{24} \]
This gives then that:
\[ (D,N_1) = \frac{k - a(p-1)}{24}. \]
Now $D$ is an effective Cartier divisor, so it corresponds to an invertible sheaf $\CO(D)$ together with a regular global section $s_D$. We then have:
\[ (D,N_1) = \deg_{\FFc_p} (\CO(D)|_{N_1}) = \sum_{x} \frac{\deg_x(s_D)}{|Aut(x)|} \]
where the sum is over the closed geometric points of the component $N_1$. Since $s_D$ is regular,  $\deg_x(s_D) \geq 0$ for each $x$, and $(D,N_1) \geq 0$. Since $p \geq 5$, it follows (see \cite{silverman}, Theorem 10.1) that for each $x$, $|Aut(x)| \leq 6$. In particular we have that if $(D,N_1) > 0$, then $(D,N_1) \geq \frac{1}{6}$.\\
First, if $k-a(p-1) = 2$, then $(D,N_1) = \frac{1}{12} < \frac{1}{6}$, which is impossible. So we must have that either $k-a(p-1)=0$ or $k-a(p-1) > 2$. Denote by $M^{a}_k$ the subset of $M_k(\GZ(p),\ZZ)$ consisting of modular forms $h$ such that $v_p(h)=0$ and $v_p(\tilde{h}) = k + a$. Define the mapping:
\[ \phi: M_{k-a(p-1)}(SL_2(\ZZ),\ZZ) \rightarrow M_k(\GZ(p),\ZZ) \]
\[g \mapsto (\Es)^a g. \]
Recall here the convention that $M_0(SL_2(\ZZ),\ZZ) = \ZZ$. It is easy to check that the image under $\phi$ of $V = M_{k-a(p-1)}(SL_2(\ZZ),\ZZ)$ lies in $M^{a}_k$. Recall also that $V$ has a Victor Miller basis, which is the unique integral basis consisting of form $f_0,\cdots, f_{d-1}$, where $d = \dim_{\QQ} M_{k-a(p-1)}(SL_2(\ZZ),\QQ)$, such that $f_i = q^i + O(q^d)$ for $0 \leq i \leq d-1$ (see for instance Proposition 6.2 in \cite{kilford}). We also adopt the convention that the Victor Miller basis of $M_0(SL_2(\ZZ),\ZZ)$ is the set $\{1\}$.
\\Assume $f$ mod $p$ is not in $ \phi(V)\otimes \FF_p$, then substracting from $f$ a suitable linear combination of the images in $ \phi(V)\otimes \FF_p$ of elements of the Victor Miller basis of $V$, we may assume that $f$ has a mod $p$ vanishing order at infinity $v_{\infty,p} (f) \geq d$, where $d = \dim_{\QQ} M_{k-a(p-1)}(SL_2(\ZZ),\QQ)$; then so does the section $s_D$ of $\CO(D)$. As the cusp infinity has only an automorphism of order $2$, we have:
\[ \sum_{x \not = \infty} \frac{\deg_x(s_D)}{|Aut(x)|}  +  \frac{v_{\infty,p}(s_D)}{2} = (D,N_1). \]
If $k-a(p-1) = 0$, then $(D,N_1) = 0$, and $d =  \dim_{\QQ} M_{k-a(p-1)}(SL_2(\ZZ),\QQ) = 1$. As $\deg_x(s_D) \geq 0$ for each $x$, this means that $\deg_x(s_D) = 0$ for all $x$, so in particular, $v_{\infty,p}(s_D) = 0$, but this contradicts the inequality $v_{\infty,p} (f) \geq d$. So assume that $k - a(p-1) > 2$. We have:
\[ \sum_{x \not = \infty} \frac{\deg_x(s_D)}{|Aut(x)|}  +  \frac{v_{\infty,p}(s_D) - d}{2} = (D,N_1) - \frac{d}{2}. \]
Consider the form $f_{d-1} \in V$, which is the element of the Victor Miller basis of $V$ with highest vanishing order at infinity, this vanishing order at infinity being $v_{\infty}(f) = d-1$. This form vanishes nowhere other than at infinity and possibly at the elliptic points of orders 2 and 3. Let $v_2(f)$ and $v_3(f)$ denote respectively the vanishing orders of $f$ at the elliptic points of orders 2 and 3, and recall that $v_2(f) \leq 1$ and $v_3(f) \leq 2$. The valence formula for level 1 modular forms (see for example Proposition 3.2 in \cite{kilford}) then gives:
\[ v_{\infty}(f_{d-1}) + \frac{1}{2}v_2(f_{d-1}) + \frac{1}{3} v_3(f_{d-1}) = \frac{k-a(p-1)}{12}, \] and therefore:
\[ \frac{k-a(p-1)}{12} - (d - 1) \leq \frac{1}{2} + \frac{2}{3} = \frac{7}{6}.\] 
Thus using the above calculation for $(D,N_1)$, we find that:
\[ \sum_{x \not = \infty} \frac{\deg_x(s_D)}{|Aut(x)|}  +  \frac{v_{\infty,p}(s_D) - d}{2} \leq \frac{1}{12},\] which forces:
\[ \sum_{x \not = \infty} \frac{\deg_x(s_D)}{|Aut(x)|}  +  \frac{v_{\infty,p}(s_D) - d}{2} = 0 \] and hence:
\[ (D,N_1) = \frac{d}{2} \]
which in turn gives:
\[ d = \frac{k-a(p-1)}{12}. \]
This however is contradicted by the dimension formula for modular forms in level 1, which says that if $k-a(p-1) \equiv 0 \pmod{12}$, then $d = 1 + \frac{k-a(p-1)}{12}$.

\end{proof}

\begin{rem}
A weaker version of Theorem \ref{congthm} can be proven by elementary methods. This can be done by adapting the argument due to Kilbourn (see \cite{kilbourn}) by which he proves Lemma \ref{killema}. This is a sketch of the argument: if $f$ satisfies the hypothesis of Theorem \ref{congthm} with $a \geq 1$, then $h = tr(f) \equiv f \pmod{p^2}$. Let $g = \frac{f - h}{p^{v_p(f-h)}}$. let $-|V : M_k(SL_2(\ZZ)) \rightarrow M_k(\GZ(p),\ZZ)$ denote the operator defined by $V(\sum a_n q^n) = \sum a_n q^{pn}$. Then one can show that $h|V \equiv p^{m - k} \tilde{g} \pmod{p}$. On the other hand, $v_p (p^{m - k} \tilde{\tilde{g}}) = v_p(p^{m-k}p^k g)\geq a+1 \geq 2$, so  $v_p (p^{m - k} \tilde{g})$ is congruent to some modular form of level 1 in weight $kp - (a+1)(p-1)$. If $w_p$ denotes the mod $p$ filtration of level $1$ modular forms, defined by:
\[ w_p(F) = \inf\{k : F \equiv G \pmod{p} \mbox{ for some } G \in M_k(SL_2(\ZZ),\ZZ)\}, \]
we know that $w_p(h|V) = pw_p(h)$ so this forces $w_p(h) \leq k - (p-1)$. It might be possible to find a variation of this argument that would give an alternative and elementary proof of Theorem \ref{congthm}. The author was however unable to find such an argument. 
\end{rem}

A simple corollary of this theorem concerns the $T$-form defined above.
\begin{cor}\label{ptcor}
$p\tilde{T} \equiv 1 \pmod{p}$.
\end{cor}
\begin{proof}
This can be proven directly from the computation of $\tilde{T}$ as in the previous section, but this follows easily from Theorem \ref{congthm}, by noting that $v_p(p\tilde{T}) = 0$, $v_p(\widetilde{p\tilde{T}}) = p$ and that $p\tilde{T}$ is of weight $p-1$. The theorem then implies that $p\tilde{T}$ is congruent mod $p$ to a modular form of weight $0$, that is, a constant, and this constant is found to be $1$ by examining the first coefficient of the $q$-expansion.
\end{proof}

\subsection{Generators of $M(\GZ(p), \ZZ)$}
In this section we identify a set of generators for $M(\GZ(p),\ZZ)$. Let $T$ denote the $T$-form defined earlier:
\[ T(z) := \left(\frac{\eta(pz)^p}{\eta(z)}\right)^2 \in M_{p-1}(\GZ(p),\ZZ).\]
Let $S$ denote the subset of $M(\GZ(p), \ZZ)$ consisting of modular forms $f$ satisfying $v_p(\tilde{f})\geq 0$. We prove the following:
\begin{thm}\label{gensZZ}
$M(\GZ(p), \ZZ)$ is generated by $T$ and $S$. 
\end{thm}
\begin{proof}
Let $f \in M_k(\GZ(p), \ZZ)$, $f \not \in S$. Put $a = -v_p(\tilde{f}) > 0$. We argue by induction on $a$. Let $g = p^a \tilde{f}$. Then $v_p(g)=0$ and $v_p(\tilde{g}) = a + v_p(\tilde{\tilde{f}}) = a + v_p(p^k f) \geq k + a$, so by Theorem \ref{congthm}, there exists $h \in M_{k-a(p-1)}(SL_2(\ZZ),\ZZ)$ such that:
\[ h \equiv g \pmod{p},\]
and by Corollary \ref{ptcor}, this can be rewritten as:
\[ (p\tilde{T})^a h\equiv g \pmod{p} \]
where now $(p\tilde{T})^a h$ and $g$ have the same weight. Thus there exists $u \in M_k(\GZ(p),\ZZ)$ such that:
\[ (p\tilde{T})^a h + pu = g.\]
Recall that $\widetilde{p\tilde{T}} = p^p T$, and that $\tilde{h} = p^{k-a(p-1)} h|V$ since $h$ is of level 1 and of weight $k-a(p-1)$. So applying the $\tilde{w}$ operator again, we get:
\[ p^{k+a} T^a h|V + p\tilde{u} = p^{k+a} f \]
and hence $p\tilde{u} = p^{k+a} v$ for some $v \in M_k(\GZ(p),\ZZ)$. Now we have:
\[ f = T^a h|V + v. \]
An easy calculation now shows that $v_p(\tilde{v}) \geq 1-a = v_p(f)$ (since $v_p(u) \geq 0$). By induction it then follows that we have the following decomposition of $f$:
\[ f = T^a f_a|V + T^{a-1} f_{a-1}|V + \cdots + T f_1|V + f_0 \]
where for each $1 \leq i \leq a$, $f_i \in M_{k - i(p-1)}(SL_2(\ZZ),\ZZ)$, and hence $v_p(\tilde{f}_i|V) = v_p(f_i) \geq 0$, and $v_p(\tilde{f_0})\geq 0$, which proves the theorem.
\end{proof}

Numerical evidence points to the following conjecture:
\begin{conj}\label{G0conj1}
The $\ZZ$-subalgebra of $M(\GZ(p),\ZZ)$ generated by $S = \{f \in M(\GZ(p),\ZZ) : v_p(\tilde{f}) \geq 0\}$ is generated in weight at most $6$. 
\end{conj}

Conjecture \ref{G0conj1} and Theorem \ref{gensZZ} together imply the following:
\begin{conj}\label{G0conj2}
The weights of the modular forms appearing in a minimal set of generators for $M(\GZ(p),\ZZ)$ are in the set $\{2,4,6,p-1\}$, and there is only one generator of weight $p-1$ (which can be chosen to be the $T$-form $T$).
\end{conj}

In Section \ref{conj2data}, we present the computational data supporting Conjecture \ref{G0conj2}. 

\section{Computational data}

\subsection{Generators and relations of $M(\GI(N), \ZZ[\frac{1}{N}])$}
We use a modification of Algorithm 1 in \cite{rustom} to calculate the structure of the algebra of $M(\GI(N), \ZZ[\frac{1}{N}])$ for $5 \leq N \leq 22$. For each $N$, the minimal set of generators picked is the union of integral bases of $M_2(\GI(N),\ZZ)$ and $M_3(\GI(N),\ZZ)$. The algorithm operates as follows:
\begin{algo}\label{algo1}\indent
\begin{enumerate}
\item $GENERATORS = \{g_1,\cdots, g_r\}$ integral basis for $M_2(\GI(N),\ZZ) \bigcup$  integral basis for $M_3(\GI(N),\ZZ) $.
\item $RELATIONS = \{ \}$.
\item for each $k \in \{4,5,6\}$:
    \begin{enumerate} \label{thecheck}
    \item $B$ = integral basis of $M_k(\Gamma, A)$.
    \item $M = \begin{pmatrix} m_1 \\ \vdots \\ m_s \end{pmatrix}$, the monomials of weight $k$ in the elements of $GENERATORS$.
    \item Express elements of $M$ as an integral linear combination of elements of $B$, obtaining an integral matrix $A$ such that $M = AB$.
    \item Calculate $D$, the Smith Normal Form of the matrix $A$, as well as the transformation matrices $U$ and $V$, such that $D = UAV$.
    \item For every diagonal entry $D_{ii}$ of $D$, check if $D_{ii}$ is invertible in $\ZZ[\frac{1}{N}]$:
    \begin{itemize}
        \item if $D_{ii}$ is not invertible in $\ZZ[\frac{1}{N}]$, then the $i$th row of $UAB$ is a relation. Add it to $RELATIONS$. 
    \end{itemize}
    \item For every row $D_j$ of $D$:
    \begin{itemize}
        \item if $D_j$ is a zero row, then the $j$th row of $UAB$ is a relation. Add it to $RELATIONS$.
    \end{itemize}
    \end{enumerate}
\item Represent each generator $g_i$ as a variable $x_i$. Find the ideal $I$ of $\ZZ[\frac{1}{N}][x_1,\cdots,x_r]$ generated by $RELATIONS$. This is the ideal of relations. Output a Gr\"{o}bner basis for $I$.
\end{enumerate}
\end{algo}
In the following table we provide the number of relations needed to generate the ideal of relations of $M(\GI(N), \ZZ[\frac{1}{N}])$, detailing the total number of generators of $M(\GI(N),\ZZ[\frac{1}{N}]$ and the number of generators in weights 2 and 3, as well as the total number of relations and the number of relations in each degree.
\[\begin{tabular}{ c | c | c |c | c | c | c | c }
  \hline
  N & generators & weight 2& weight 3 & relations & degree 4  & degree 5 & degree 6 \\
  \hline       
5 & 7 & 3 & 4 & 17 & 1 & 6 & 10 \\ 
6 & 7 & 3 & 4 & 17 & 1 & 6 & 10 \\ 
7 & 12 & 5 & 7 & 58 & 6 & 24 & 28 \\ 
8 & 12 & 5 & 7 & 58 & 6 & 24 & 28 \\ 
9 & 17 & 7 & 10 & 124 & 15 & 54 & 55 \\ 
10 & 17 & 7 & 10 & 124 & 15 & 54 & 55 \\ 
11 & 25 & 10 & 15 & 281 & 35 & 125 & 121 \\ 
12 & 22 & 9 & 13 & 215 & 28 & 96 & 91 \\ 
13 & 33 & 13 & 20 & 502 & 64 & 226 & 212 \\ 
14 & 30 & 12 & 18 & 412 & 54 & 186 & 172 \\ 
15 & 40 & 16 & 24 & 749 & 104 & 344 & 301 \\ 
16 & 38 & 15 & 23 & 673 & 89 & 306 & 278 \\ 
17 & 52 & 20 & 32 & 1281 & 166 & 584 & 531 \\ 
18 & 43 & 17 & 26 & 869 & 118 & 398 & 353 \\ 
19 & 63 & 24 & 39 & 1902 & 246 & 867 & 789 \\ 
20 & 56 & 22 & 34 & 1495 & 207 & 690 & 598 \\ 
21 & 72 & 28 & 44 & 2497 & 346 & 1156 & 995 \\ 
22 & 65 & 25 & 40 & 2027 & 270 & 930 & 827 \\

  \hline  
\end{tabular}\]

\subsection{Generators of $M(\GZ(p),\ZZ)$}\label{conj2data}
Recall that the set $S$ is defined as the set of modular forms $f \in M(\GZ(p),\ZZ)$ such that $v_p(\tilde{f})\geq 0$. Recall also the $T$-form $T(z) := \left(\frac{\eta(pz)^p}{\eta(z)}\right)^2$. We will provide computational evidence for Conjecture \ref{G0conj2}.\\
We use Algorithm 1 in \cite{rustom} to calculate the degrees of generators in a minimal set of generators for $M(\GZ(p),\ZZ)$. The following table details the results.

\[\begin{tabular}{ l | c | r }
  \hline
  $p$ & weights of generators\\
  \hline       
  5 &  2, 4, 4 & --\\
  7 & 2, 4, 4, 6, 6 & --\\
  11 & 2, 2, 4, 6, 10 & --\\
  13 & 2, 4, 4, 4, 4, 6, 6, 12 & --\\
  17 & 2, 2, 4, 4, 4, 6, 16 & --\\
  19 & 2, 2, 4, 4, 4, 6, 6, 18 & --\\
  23 & 2, 2, 2, 4, 4, 6, 22 & --\\
  29 & 2, 2, 2, 4, 4, 4, 4, 6, 28 & --\\
  31 & 2, 2, 2, 4, 4, 4, 4, 6, 6, 30 & --\\

  \hline  
\end{tabular}\]
In each level $p$ above, the form in weight $p-1$ can be chosen to be the $T$-form. 

\bibliographystyle{amsalpha}
\bibliography{bibliorelZZ}

\def\cprime{$'$}
\providecommand{\bysame}{\leavevmode\hbox to3em{\hrulefill}\thinspace}
\providecommand{\MR}{\relax\ifhmode\unskip\space\fi MR }
\providecommand{\MRhref}[2]{%
  \href{http://www.ams.org/mathscinet-getitem?mr=#1}{#2}
}
\providecommand{\href}[2]{#2}
\begin{thebibliography}{Mum70}

\bibitem[BDP]{BDP}
Massimo Bertolini, Henri Darmon, and Kartik Prasanna, \emph{{{$p$}-adic
  L-functions and the coniveau filtration on Chow groups}}, preprint,
  \url{http://www.math.mcgill.ca/darmon/pub/Articles/Research/60.BDP5-coniveau/paper.pdf}.

\bibitem[Del75]{delignetate}
P.~Deligne, \emph{Courbes elliptiques: formulaire d'apr\`es {J}. {T}ate},
  Modular functions of one variable, {IV} ({P}roc. {I}nternat. {S}ummer
  {S}chool, {U}niv. {A}ntwerp, {A}ntwerp, 1972), Springer, Berlin, 1975,
  pp.~53--73. Lecture Notes in Math., Vol. 476. \MR{0387292 (52 \#8135)}

\bibitem[DM69]{DM}
P.~Deligne and D.~Mumford, \emph{The irreducibility of the space of curves of
  given genus}, Inst. Hautes \'Etudes Sci. Publ. Math. (1969), no.~36, 75--109.
  \MR{0262240 (41 \#6850)}

\bibitem[DR73]{DR}
P.~Deligne and M.~Rapoport, \emph{Les sch\'emas de modules de courbes
  elliptiques}, Modular functions of one variable, {II} ({P}roc. {I}nternat.
  {S}ummer {S}chool, {U}niv. {A}ntwerp, {A}ntwerp, 1972), Springer, Berlin,
  1973, pp.~143--316. Lecture Notes in Math., Vol. 349.

\bibitem[Gro90]{G90}
Benedict~H. Gross, \emph{A tameness criterion for {G}alois representations
  associated to modular forms (mod {$p$})}, Duke Math. J. \textbf{61} (1990),
  no.~2, 445--517.

\bibitem[Kil07]{kilbourn}
Timothy Kilbourn, \emph{Congruence properties of {F}ourier coefficients of
  modular forms}, ProQuest LLC, Ann Arbor, MI, 2007, Thesis (Ph.D.)--University
  of Illinois at Urbana-Champaign. \MR{2711108}

\bibitem[Kil08]{kilford}
L.~J.~P. Kilford, \emph{Modular forms}, Imperial College Press, London, 2008, A
  classical and computational introduction. \MR{2441106 (2009m:11001)}

\bibitem[K{\"o}h11]{koehler}
G{\"u}nter K{\"o}hler, \emph{Eta products and theta series identities},
  Springer Monographs in Mathematics, Springer, Heidelberg, 2011.

\bibitem[Miy06]{miyake}
Toshitsune Miyake, \emph{Modular forms}, english ed., Springer Monographs in
  Mathematics, Springer-Verlag, Berlin, 2006, Translated from the 1976 Japanese
  original by Yoshitaka Maeda. \MR{2194815 (2006g:11084)}

\bibitem[Mum70]{mumford}
David Mumford, \emph{Varieties defined by quadratic equations}, Questions on
  {A}lgebraic {V}arieties ({C}.{I}.{M}.{E}., {III} {C}iclo, {V}arenna, 1969),
  Edizioni Cremonese, Rome, 1970, pp.~29--100. \MR{0282975 (44 \#209)}

\bibitem[Rus12]{rustom}
Nadim Rustom, \emph{{Generators of graded algebras of modular forms}}, to
  appear in Journal of Number Theory. Preprint at
  \url{http://arxiv.org/abs/1209.3864}.

\bibitem[Sch79]{scholl}
A.~J. Scholl, \emph{On the algebra of modular forms on a congruence subgroup},
  Math. Proc. Cambridge Philos. Soc. \textbf{86} (1979), no.~3, 461--466.

\bibitem[Ser73]{serre}
Jean-Pierre Serre, \emph{Formes modulaires et fonctions z\^eta {$p$}-adiques},
  Modular functions of one variable, {III} ({P}roc. {I}nternat. {S}ummer
  {S}chool, {U}niv. {A}ntwerp, 1972), Springer, Berlin, 1973, pp.~191--268.
  Lecture Notes in Math., Vol. 350. \MR{0404145 (53 \#7949a)}

\bibitem[Sil09]{silverman}
Joseph~H. Silverman, \emph{The arithmetic of elliptic curves}, second ed.,
  Graduate Texts in Mathematics, vol. 106, Springer, Dordrecht, 2009.
  \MR{2514094 (2010i:11005)}

\bibitem[SS11]{TS11}
Hayato Saito and Tomohiko Suda, \emph{An explicit structure of the graded ring
  of modular forms of small level (pre-print)}, arXiv:1108.3933v3 [math.NT]
  (2011).

\bibitem[Wag80]{wagreich1}
Philip Wagreich, \emph{Algebras of automorphic forms with few generators},
  Trans. Amer. Math. Soc. \textbf{262} (1980), no.~2, 367--389. \MR{586722
  (82e:10044)}

\bibitem[Wag81]{wagreich2}
\bysame, \emph{Automorphic forms and singularities with {${\bf C}^{\ast}
  $}-action}, Illinois J. Math. \textbf{25} (1981), no.~3, 359--382. \MR{620423
  (82m:10045)}

\end{thebibliography}

Nadim Rustom, Department of Mathematical Sciences, University of Copenhagen, Universitetsparken 5, 2100 Copenhagen \O\space Denmark\\ \texttt{rustom@math.ku.dk}
\end{document}